\documentclass[reqno, 12pt,a4letter]{amsart}
\usepackage{amsmath,amsxtra,amssymb,latexsym, amscd,amsthm}
\usepackage[utf8]{inputenc}
\usepackage[mathscr]{euscript}
\usepackage{mathrsfs}
\usepackage[english]{babel}
\usepackage{color}
\usepackage[active]{srcltx}
\newtheorem {theorem}{Theorem}[section]
\newtheorem {claim}{Claim}[section]
\newtheorem {corollary}{Corollary}[section]
\newtheorem {lemma}{Lemma}[section]

\newtheorem {definition}{Definition}[section]
\newtheorem {remark}{Remark}[section]
\newtheorem {proposition}{Proposition}[section]

\def\EES{{\accent"5E e}\kern-.5em\raise.8ex\hbox{\char'23 }}
\def\ow{o\kern-.42em\raise.82ex\hbox{
   \vrule width .12em height .0ex depth .075ex \kern-0.16em \char'56}\kern-.07em}
\def\OW{o\kern-.460em\raise1.36ex\hbox{
\vrule width .13em height .0ex depth .075ex \kern-0.16em
\char'56}\kern-.07em}
\def\DD{D\kern-.7em\raise0.4ex\hbox{\char '55}\kern.33em}

\title{\L ojasiewicz inequalities with explicit exponent for smallest singular value functions}

\author{S\~i-Ti\d{\^e}p \DD INH$^\dagger$}
\address{Institute of Mathematics, VAST, 18, Hoang Quoc Viet Road, Cau Giay District 10307, Hanoi, Vietnam}
\email{dstiep@math.ac.vn}

\author{Ti\EES n-S\OW n Ph\d{A}m$^\ddagger$}
\address{Department of Mathematics, University of Dalat, 1 Phu Dong Thien Vuong, Dalat, Vietnam}
\email{sonpt@dlu.edu.vn}



\date{\today}

\begin{document}
\maketitle

\begin{abstract}
Let $F(x) := (f_{ij}(x))_{i=1,\ldots,p; j=1,\ldots,q},$ be a ($p\times q$)-real polynomial matrix and let $f(x)$ be the smallest singular value function of $F(x).$ 
In this paper, we first give the following {\em nonsmooth} version of \L ojasiewicz gradient inequality for the function $f$ with an explicit exponent: {\em For any $\bar x\in \Bbb R^n$, there exist $c > 0$ and $\epsilon > 0$ such that we have for all $\|x - \bar{x}\| < \epsilon,$
\begin{equation*}
\inf \{ \| w \| \ : \ w \in {\partial} f(x) \} \ \ge \ c\, |f(x)-f(\bar x)|^{1 - \frac{2}{\mathscr R(n+p,2d+2)}},
\end{equation*}
where ${\partial} f(x)$ is the limiting subdifferential of $f$ at $x$, $d:=\max_{i=1,\ldots,p; j=1,\ldots,q}\deg f_{i j}$ and  $\mathscr R(n, d) := d(3d - 3)^{n-1}$ if $d \ge 2$ and $\mathscr R(n, d) := 1$ if $d = 1.$} Then we establish some versions of \L ojasiewicz inequality for the distance function with explicit exponents, locally and globally, for the smallest singular value function $f(x)$ of the matrix $F(x)$.
\end{abstract}

\pagestyle{plain}

\tableofcontents

\section{Introduction}
Let $f$ be a real analytic function in a neighborhood of $\bar x$ in $\Bbb R^n$. Then the \L ojasiewicz gradient inequality~(\cite{Lojasiewicz1958, Lojasiewicz1965, Lojasiewicz1984}) states that there exist some constants $c>0$ and $\alpha\in [0, 1)$ such that, in a neighborhood of $\bar x$, we have
\begin{equation}\label{ClassicalLojasiewiczGradient}\|\nabla f(x)\|\ge c|f(x)-f(\bar x)|^\alpha.\end{equation}
The {\it \L ojasiewicz exponent} of $f$ at $\bar x$, denoted by $\alpha_f$, is the infimum of the exponents $\alpha$ for which Inequality~(\ref{ClassicalLojasiewiczGradient}) holds.

If $f$ is a real polynomial in $n$ variables and of degree $d$, D'Acunto and Kurdyka~\cite{Acunto2005} proved that $\alpha_f$ is bounded from above by a constant depending only on $n$ and $d$. Precisely, the following holds.

\begin{theorem} [\cite{Acunto2005}] \label{GradientInequality}
Let $f \colon \mathbb{R}^n \rightarrow \mathbb{R}$ be a polynomial function of degree $d$. Then there are some positive constants $c$ and $\epsilon$ such that
$$\|\nabla f(x)\| \ge c\, |f(x)-f(\bar x)|^{1 - \frac{1}{\mathscr R(n,d)}} \text{ for all }  x  \text{ such that } \|x-\bar x\|<\epsilon,$$
where
\begin{equation}\label{Eq1}
\mathscr R(n, d) :=
\begin{cases}
d(3d - 3)^{n-1} & \text{ if } \ d \ge 2,\\
1 & \text { if } \ d = 1.
\end{cases}
\end{equation}
\end{theorem}
Consequently, the following \L ojasiewicz inequality with an explicit exponent for the distance function can be deduced easily the same way as in \cite{Lojasiewicz1984}: There are some positive constants $c$ and $\epsilon$ such that
$$|f(x)-f(\bar x)| \ge c\, \textrm{dist}(x,f^{-1}(\bar x))^{\frac{1}{\mathscr R(n,d)}} \quad \text{ for all } \quad x \quad \text{such that} \quad \|x-\bar x\|<\epsilon,$$
where $\textrm{dist}(\cdot, \cdot)$ denotes the Euclidean distance function.

Note that knowing the \L ojasiewicz exponent is important in theory and application (see \cite{Absil2006, Haraux2012, Kurdyka2000, Lojasiewicz1984, Nie2007, Schweighofer2004}). In the polynomial case, as far as we know, $1 - \frac{1}{\mathscr R(n,d)}$ is the best upper bound for $\alpha_f$.

In the case $f$ is semialgebraic continuous, which is not necessary smooth, inequalities of type~(\ref{ClassicalLojasiewiczGradient}) still exist if we replace $\|\nabla f(x)\|$ by $\frak m_f(x)$, which is the {\em nonsmooth slope} of $f$ at $x$ (see Definition~\ref{NonSmoothSlope}). In fact, inequalities of type ${\frak m}_f(x)\ge c|f(x)-f(\bar x)|^\alpha$ exist in a more general context, however, calculating or just giving an explicit upper bound for the \L ojasiewicz exponent, in general, is quite delicate.

In this paper, we propose a version of \L ojasiewicz gradient inequality with explicit exponent for the smallest singular value function of a given real polynomial matrix. We prove that the \L ojasiewicz exponent is bounded from above by a constant depending only on the degrees, the number of variables of the polynomials and the number of rows of the matrix. Precisely, our main result is the following.

\begin{theorem} \label{NonSmoothTheorem}
Let $\mathscr M(p,q)$ be the space of $(p \times q)$-matrices. Let $F \colon  \Bbb R^n \rightarrow \mathscr M(p,q),\ x\mapsto F(x)=(f_{ij}(x))_{i=1,\ldots,p;j=1,\ldots,q},$ be a $(p \times q)$-polynomial matrix such that $p \le q$ and $d := \max_{i =1, \ldots, p; j = 1, \ldots, q}\deg f_{i j}>0$. Let $f(x)$ be the corresponding smallest singular value function. Let $\bar x \in\Bbb R^n$. Then there exist some positive constants $c$ and $\epsilon$ such that
\begin{equation*}
\inf \{ \| w \| \ : \ w \in {\partial} f(x) \} \ \ge \ c\, |f(x)-f(\bar x)|^{1 - \frac{1}{\mathscr R(n+p,2d+2)}} \quad \text{ for all } x \in \mathbb{R}^n,  \|x-\bar x\|<\epsilon.
\end{equation*}
\end{theorem}

The principal idea of the proof of Theorem~\ref{NonSmoothTheorem} is to produce a polynomial $g$ such that the limiting subdifferential of $f$ can be related to $\nabla g,$ and hence the nonsmooth slope ${\frak m}_f$ of $f$ can be estimated via $\|\nabla g\|$. Note that, recently, we have proved a similar result for largest eigenvalue functions of real symmetric matrices (see \cite{Dinh2016}), the techniques used in \cite{Dinh2016} permits also to prove Theorem \ref{NonSmoothTheorem} by considering the matrix $-F(x)F^T(x)$ and its corresponding maximal eigenvalue function $-f^2(x)$. However, the techniques used in this paper yields a much better exponent.

As a consequence of Theorem~\ref{NonSmoothTheorem}, we give a local (Proposition~\ref{CompactErrorBound}) and a global (Corollary~\ref{GlobalKollar}) version of \L ojasiewicz inequality for smallest singular value functions which bound the distance to the zero set by some explicit positive power of the function. Moreover, we give some versions of separation of semialgebraic sets associated to smallest singular value functions with explicit exponents (Corollaries~\ref{Separation} and~\ref{GlobalSeparation}). A global version of \L ojasiewicz inequality for tame singular value functions will be also established (Proposition~\ref{HolderTypeTheorem}).

The paper is organized as follows. Section~\ref{NonsmoothSlope} recalls some basic notions and results of nonsmooth analysis. The proof of Theorem~\ref{NonSmoothTheorem} is given in Section~\ref{ProofMain}. Sections~\ref{LocalLojasiewicz},~\ref{GlobalLojasiewicz}~and~\ref{TameLojasiewicz} contain some consequences of Theorem~\ref{NonSmoothTheorem} on various types of \L ojasiewicz inequality in local and global setups.

Throughout this paper, we denote by $\textrm{dist}(\cdot, \cdot)$ the Euclidean distance function. We denote $\mathbb{B}^n(x, \epsilon),\ \overline{\mathbb{B}}^n(x, \epsilon),$ and ${\mathbb{S}}^{n - 1}(x, \epsilon)$, respectively, the open ball, the closed ball and the sphere centered at $x$, of radius $\epsilon > 0$ in the Euclidean space ${\mathbb{R}}^{n}.$ In the case where $x = 0$ and $\epsilon = 1,$ we write $\mathbb{B}^n,\ \overline{\mathbb{B}}^n$ and ${\mathbb{S}}^{n - 1}$, respectively, instead of $\mathbb{B}^n(0, 1),\ \overline{\mathbb{B}}^n(0, 1)$ and ${\mathbb{S}}^{n - 1}(0, 1).$

\section{Nonsmooth slope}\label{NonsmoothSlope}
We first recall the notion of limiting subdifferential, that is, an appropriate multivalued operator playing the role of the usual gradient map. For nonsmooth analysis we refer to the comprehensive texts~\cite{Mordukhovich2006, Rockafellar1998}.

\begin{definition}{\rm
\begin{enumerate}
  \item[(i)] The {\em Fr\'echet subdifferential} $\hat{\partial} f(x)$ of a continuous function $f \colon {\Bbb R}^n \rightarrow {\Bbb R}$ at $x \in {\Bbb R}^n$ is given by
$$\hat{\partial} f(x) := \left \{ v \in {\Bbb R}^n \ : \ \liminf_{\| h \| \to 0, \ h \ne 0} \frac{f(x + h) - f(x) - \langle v, h \rangle}{\| h \|} \ge 0 \right \}.$$
  \item[(ii)] The {\em limiting subdifferential} at $x \in {\Bbb R}^n,$ denoted by ${\partial} f(x),$ is the set of all cluster points of sequences $\{v^k\}_{k \ge 1}$ such that $v^k\in \hat{\partial} f(x^k)$ and $x^k \to x$ as $k \to \infty.$
\end{enumerate}
}\end{definition}

The next lemma is well-known (see e.g., \cite{Mordukhovich2006, Rockafellar1998}).
\begin{lemma} \label{Lemma21}
Let $f \colon {\Bbb R}^n \rightarrow {\Bbb R}$ be a continuous function. The following statements hold:
\begin{enumerate}
\item[{\rm (i)}] If $x \in {\Bbb R}^n$ is a local (or global) minimum of $f$ then $0 \in {\partial} f(x).$

\item[{\rm (ii)}] Let $g \colon {\Bbb R}^n \rightarrow {\Bbb R}$ be a locally Lipschitz function. Then
$${\partial} (f + g)(x) \subseteq \partial f(x) + \partial g(x).$$

\item[(iii)] $\partial\|x\| =
\begin{cases}
\frac{x}{\|x\|} & \text{ if } \ x\not=0,\\
\overline {\Bbb B}^n & \text { if }\  x=0.
\end{cases}$
\end{enumerate}
\end{lemma}

\begin{definition}\label{NonSmoothSlope}{\rm
Using the limiting subdifferential $\partial f,$ we define the {\em
nonsmooth slope} of $f$ by
$${\frak m}_f(x) := \inf \{ \|v\| \ : \ v \in {\partial} f(x) \}.$$
By definition, ${\frak m}_f(x) = + \infty$ whenever ${\partial} f(x) = \emptyset.$
}\end{definition}

\begin{remark}{\rm
(i) It is a well-known result of variational analysis that
$\hat{\partial} f(x)$ (and a fortiori $\partial f(x)$) is not empty in a dense subset of the domain of $f$ (see~\cite{Rockafellar1998}, for example).

(ii) If the function $f$ is of class $C^1,$ the above notions coincides
with the usual concept of gradient; that is, ${\partial} f(x) =
\hat{\partial} f(x) = \{\nabla f(x) \},$ and hence ${\frak m}_f(x) =
\|\nabla f(x) \|.$

(iii) By Tarski--Seidenberg Theorem (see~\cite{Bochnak1998}), it is not hard to show that if the function $f$ is semi-algebraic then so is ${\frak m}_f.$
}\end{remark}

The following lemma will be useful in the sequel (see \cite[Corollary~2]{Ngai2009}).

\begin{lemma} \label{Lemma22}
Let $f \colon \mathbb{R}^n \rightarrow \mathbb{R}$ be a continuous function and let $\bar{x} \in \mathbb{R}^n$ be such that $f(\bar{x}) = 0.$ Let $S := \{x \in \mathbb{R}^n \ | \ f(x) \le 0\}.$ Assume that there are real numbers $c > 0, \delta > 0,$ and $\alpha \in [0, 1),$ such that 
$${\frak m}_f(x) \ge c |f(x)|^{\alpha} \quad \textrm{ for all } \quad \|x - \bar{x}\| \le \delta \ \textrm{ and } \ x \not \in S.$$
Then we have
\begin{eqnarray*}
\big[f(x)\big]_+^{1 - \alpha} &\ge& c(1 - \alpha)  \, \mathrm{dist}(x, S) \quad 
\mbox{ whenever } \quad \|x - \bar{x}\| \le \frac{\delta}{2},
\end{eqnarray*}
where $[f(x)]_+ := \max\{f(x), 0\}$ and $\mathrm{dist}(x, S)$ denotes the Euclidean distance from $x$ to $S.$ 
\end{lemma}

\section{Proof of the main result}\label{ProofMain}


First of all, it is not hard to see that the smallest singular value function can be expressed by the following formula
$$f(x) = \min_{y \in \mathbb{S}^{p - 1}} \left \| \sum_{i = 1}^p y_i F_i(x) \right\|,$$
where $F_i(x)$ stands for the $i^\text{th}$ row of the matrix $F(x).$ Let us define the function $g \colon \mathbb{R}^n \times \mathbb{R}^p \rightarrow \mathbb{R}, (x, y) \mapsto g(x, y),$ by
$$g(x, y) := \sum_{i = 1}^p \sum_{j = 1}^p  y_i y_j \langle F_i(x), F_j(x) \rangle-\sum_{i = 1}^p[f(\bar x)]^2y_i^2.$$
Clearly, $g$ is a polynomial in $n + p$ variables of degree at most $2d + 2$, recall that $d = \max_{i =1, \ldots, p; j = 1, \ldots, q}\deg f_{i j}.$ Let
$$\tilde{f}(x) := \min_{y \in \mathbb{S}^{p - 1}} g(x, y).$$
Then $\tilde{f}(x)=[f(x)]^2-[f(\bar x)]^2$ is the minimal value function of the matrix $F(x)F^T(x)-[f(\bar x)]^2I_p$, where $F^T(x)$ denotes the transpose of $F(x)$ and $I_p$ is the unit matrix of order $p$.
\begin{claim}\label{Claim1}
$\tilde f$ and $f$ are locally Lipschitz semi-algebraic functions.
\end{claim}
\begin{proof}
Lipschitz continuity of the functions $\tilde f$ and $f$ follows immediately from definitions. Thanks to Tarski--Seidenberg principle (see e.g., \cite{Seidenberg1954, Tarski1931, Tarski1951}), these are semi-algebraic.
\end{proof}

For each $x \in \mathbb{R}^n$, we put
$$E(x) := \{y \in \mathbb{S}^{p - 1} \ : \ \tilde{f}(x) = g(x, y)\}.$$
Since the sphere $\mathbb{S}^{p - 1}$ is compact, $E(x)$ is a nonempty and compact set for all $x \in \mathbb{R}^n.$ Moreover, we have

\begin{claim} \label{Claim2}
The set-valued  map $E \colon \mathbb{R}^{n} \rightrightarrows \mathbb{R}^n, x \mapsto E(x),$ is locally H\"older stable; i.e., for any fixed $\bar{x} \in \mathbb{R}^{n}$ and $\epsilon>0$, there exist some positive constants $c$ and $\alpha$ such that
$$E(x) \subset E(\bar{x}) + c \|x - \bar{x}\|^{\alpha}\, \overline{\mathbb{B}}^p \quad \textrm{ for all } \quad x \in \Bbb{B}^n(\bar x,\epsilon).$$
\end{claim}
\begin{proof} Let
\begin{displaymath}
\begin{array}{lcll}
$$H \colon &\mathbb{R}^n \times \mathbb{R}^p &\rightarrow &\mathbb{R}$$\\
$$         &(x, y)                           &\mapsto     &H(x,y):=|g(x,y)-\tilde f(x)|+\displaystyle\left|\sum_{i=1}^p y_i^2-1\right|.$$
\end{array}
\end{displaymath}
It is easy to check that the function $H$ is semi-algebraic and locally Lipschitz. Further, we have
\begin{displaymath}
\begin{array}{lllll}
$$ E(x)  & =  & \{y \in \Bbb S^{p-1} & : & g(x,y)-\tilde f(x) = 0\}$$ \\
$$       & =  & \{y \in \mathbb{R}^p & : & \displaystyle\sum_{i=1}^p y_i^2-1=0, \ g(x,y)-\tilde f(x) = 0 \}$$ \\
$$       & =  & \{y \in \mathbb{R}^p & : & H(x,y) = 0 \}.$$
\end{array}
\end{displaymath}
Since the sphere $\Bbb S^{p-1}$ is compact, it follows from the \L ojasiewicz inequality (see, for
example,~\cite{Bochnak1998}) that there are some constants $c > 0$ and $\alpha > 0$ such that
\begin{eqnarray*}
c\, \textrm{dist}(y , E(\bar{x})) & \le & |H(\bar{x}, y)|^\alpha \quad \textrm{ for all } \quad y \in \Bbb{S}^{p-1}.
\end{eqnarray*}
On the other hand, since the function $H$ is locally Lipschitz, it is globally Lipschitz on the compact set $\overline{\Bbb{B}}^n(\bar{x},\epsilon) \times \Bbb{S}^{p-1};$  in particular, there exists a constant $L > 0$ such that
 \begin{eqnarray*}
|H(x, y) - H(\bar{x}, y)|  & \le & L \|x - \bar{x}\| \quad \textrm{ for all } \quad (x, y) \in \mathbb{B}^n(\bar{x},\epsilon) \times \mathbb{S}^{p-1}.
\end{eqnarray*}

Let $x \in \mathbb{B}^n(\bar{x},\epsilon)$ and take an arbitrary  $y \in E(x).$ Then $H(x, y) = 0.$ Therefore,
 \begin{eqnarray*}
c\, \textrm{dist}(y , E(\bar{x}))
& \le & |H(\bar{x}, y)|^\alpha \\
& = & |H(x, y) - H(\bar{x}, y)|^\alpha \\
& \le & L^\alpha \|x - \bar{x}\|^\alpha .
 \end{eqnarray*}
This implies immediately the required statement.
\end{proof}

\begin{claim} \label{Claim3}
For all ${x} \in \mathbb{R}^n$ and all ${y} \in E({x}),$ the following statements hold:
\begin{enumerate}
\item[{\rm (i)}] $\tilde{f}({x}) = g({x}, {y}).$
\item[{\rm (ii)}] $\hat{\partial} \tilde{f}({x}) \subset \{\nabla_xg({x}, {y})\}.$ In particular, $\emptyset \ne {\partial} \tilde{f}({x}) \subset \cup_{z \in E(x)}\{\nabla_xg({x}, {z})\}$ and 
$${\frak m}_{\tilde{f}}({x}) \ge \inf_{z \in E(x)}\|\nabla_xg({x}, {z})\|.$$
\item[{\rm (iii)}] $\nabla_y g({x}, {y}) - 2 {\tilde{f}}({x}) {y} = 0.$
\end{enumerate}
\end{claim}
\begin{proof}
(i) Clearly.

(ii) Take arbitrary $v \in \hat{\partial} {\tilde{f}}({x}).$ By definition, for every $\epsilon > 0$ there exists $\delta > 0$ such that
$${\tilde{f}}({x} + h) - {\tilde{f}}({x}) - \langle v, h\rangle \ge - \epsilon \|h\|, \quad \textrm{ for all } \quad h \in \mathbb{B}^n(0, \delta).$$
Define the function $\phi \colon \mathbb{R}^n \rightarrow \mathbb{R}, h \mapsto \phi(h),$ by
$$\phi(h) := g({x} + h, {y}) - \langle v, h\rangle + \epsilon \|h\|.$$
We have for all $h \in \mathbb{B}^n(0, \delta),$
\begin{eqnarray*}
\phi(h)
&\ge& {\tilde{f}}({x} + h) - \langle v, h\rangle + \epsilon \|h\| \\
&\ge& {\tilde{f}}({x}) = g({x}, {y}) = \phi(0).
\end{eqnarray*}
Consequently, $0$ is a local minimum of $\phi.$ Then by Claim~\ref{Lemma21}, we have
$$0 \in \partial \phi(0) \subseteq \nabla_x g({x}, {y}) - v + \epsilon \overline{\mathbb{B}}^n.$$
Therefore, $\|\nabla_x g({x}, {y}) - v\| \le \epsilon.$ Letting $\epsilon \to 0$ yields $v = \nabla_x g({x}, {y}).$
Since this equality holds for all $v \in \hat{\partial} {\tilde{f}}({x}),$ we obtain
$\hat{\partial} {\tilde{f}}({x}) \subset \{\nabla_xg({x}, {y})\}.$

On the other hand, since the function $\tilde{f}$ is semialgebraic, it follows from Cell Decomposition Theorem (see~\cite{Bochnak1998, Dries1996}) that $f$ is of class $C^1$ on a semi-algebraic open dense set $U \subset \mathbb{R}^n.$ Then we have for all $x \in U,$
$$\hat{\partial} {\tilde{f}}({x}) = \partial {\tilde{f}}({x}) = \{\nabla_xg({x}, {y})\}.$$
Note that the function $g$ is of class $C^\infty$ and the set $E(x)$ is compact. Therefore, by Claim~\ref{Claim2} and by definition, we get that 
$\emptyset \ne {\partial} \tilde{f}({x}) \subset \cup_{z \in E(x)}\{\nabla_xg({x}, {z})\}$ for all $x \in \mathbb{R}^n.$

(iii) By definition, $\nabla_y g(x, y) = 2 F(x) F(x)^T y-2[f(\bar x)]^2y.$  Hence
$$\langle \nabla_y g({x}, {y}), {y} \rangle = \langle 2 F(x) F(x)^T y-2[f(\bar x)]^2y, y \rangle = 2 g({x}, {y}) = 2 {\tilde{f}}({x}).$$

On the other hand, be definition, we have
$$g({x}, {y}) = \tilde f(x) = \min_{\|z\|^2 = 1} g(x, z).$$
Thanks to Lagrange's multiplier theorem, there exists $\lambda \in \mathbb{R}$ such that
$$\nabla_y g({x}, {y}) - 2 \lambda {y} = 0.$$
Therefore
$$2 \lambda = 2 \lambda  \langle {y}, {y} \rangle = \langle \nabla_y g({x}, {y}), {y} \rangle = 2 {\tilde{f}}({x}),$$
which completes the proof.
\end{proof}


\begin{claim} \label{Claim4}
There exist some positive constants $c$ and $\epsilon'$ such that
\begin{equation*}
\|\nabla g(x, y)\| \ge c\, |g(x, y)-\tilde f(\bar x)|^{1 - \frac{1}{\mathscr R(n + p, 2d + 2)}}
\end{equation*}
for all $x \in \mathbb{B}^n(\bar{x}, \epsilon')$ and all $y \in \mathbb{R}^{p}$ with ${\rm dist}(y, E(\bar{x})) < \epsilon'.$
\end{claim}
\begin{proof}
For any $\bar{y} \in E(\bar{x}),$ we have $\tilde f(\bar x)=g(\bar{x}, \bar{y})$. So by Theorem~\ref{GradientInequality},
there exist some positive constants $c(\bar{y})$ and $\epsilon(\bar{y})$  such that
\begin{equation*}
\|\nabla g(x, y)\| \ge c(\bar{y}) \, |g(x, y)-\tilde f(\bar x)|^{1 - \frac{1}{\mathscr R(n + p, 2d + 2)}} \quad
\textrm{ for } \quad \|(x, y)  -  (\bar{x}, \bar{y}) \| < \epsilon(\bar{y}).
\end{equation*}

Clearly, $E(\bar{x}) \subset \bigcup_{\bar{y} \in E(\bar{x})} \mathbb{B}^p \left (\bar{y}, \frac{\epsilon(\bar{y})}{2} \right).$ Since $E(\bar{x})$ is a compact set, there exist finite points $\bar{y}^k \in E(\bar{x}),$  for $k = 1, \ldots, N,$ such that
$$E(\bar{x}) \subset \bigcup_{k = 1}^N \mathbb{B}^p \left (\bar{y}^k, \frac{\epsilon(\bar{y}^k)}{2} \right).$$
Then the constants $c := \min_{k = 1, \ldots, N} c(\bar{y}^k)$ and
$\epsilon' := \min_{k = 1, \ldots, N} \frac{\epsilon(\bar{y}^k)}{2}$ have the desired properties.
\end{proof}

\begin{claim} \label{Claim5}
For each  $\epsilon' > 0$ there exists a positive constant $\epsilon < \epsilon'$ such that for all $x \in \mathbb{B}^n(\bar{x}, \epsilon)$ and all $y \in E(x),$ we have
$$\mathrm{dist}(y, E(\bar{x}))  < \epsilon'.$$
\end{claim}
\begin{proof}
By contradiction, assume that there exist a number $\epsilon' > 0$ and some sequences $x^k \in \mathbb{R}^n$ and $y^k\in E(x^k)$ such that $\lim_{k \to \infty} x^k  = \bar{x}$ and
$$\mathrm{dist}(y^k, E(\bar{x})) \ge \epsilon'.$$
Since $E(x^k)$ is a subset of the compact set $\mathbb{S}^{p - 1},$ we may assume that the limit $\bar{y} := \lim_{k \to \infty} y^k \in \mathbb{S}^{p - 1}$ exists. Note that $\tilde{f}(x^k) = g(x^k, y^k).$ Hence, by continuity, we get $\tilde{f}(\bar{x}) = g(\bar{x}, \bar{y}),$ and so $\bar{y} \in E(\bar{x}),$ which is a contradiction.
\end{proof}

\begin{claim} \label{Claim6} 
Then there exist some positive constants $c$ and $\epsilon$ such that
\begin{equation*}
\frak m_{\tilde{f}}(x)\ge c\, |{\tilde{f}}(x)|^{1 - \frac{1}{\mathscr R(n + p, 2d + 2)}} \quad \text{ for all } \quad x \in \mathbb{B}^n(\bar{x}, \epsilon).
\end{equation*}
\end{claim}
\begin{proof}
Let $c, \epsilon',$ and $\epsilon < \epsilon'$ be some positive constants such that Claims~\ref{Claim4}~and~\ref{Claim5} hold. Take arbitrary $x \in \mathbb{B}^n(\bar{x}, \epsilon ).$
Since the set $E(x)$ is compact, there exists a point $y \in E(x)$ such that $ \|\nabla_x g(x, y)\| = \inf_{z \in E(x)}  \|\nabla_x g(x, z)\|.$ It follows from Claim~\ref{Claim3} that
\begin{eqnarray*}
{\frak m}_{\tilde{f}}(x) &\ge& \|\nabla_x g(x, y)\| \quad \textrm{ and } \quad \nabla_y g(x, y) \ = \ 2 {\tilde{f}}(x) y.
\end{eqnarray*}
Since, all norms on normed vector spaces of finite dimension are equivalent, there is a constant $c_1>0$ such that
\begin{eqnarray*}
c_1\|\left (\nabla_x g(x, y), \nabla_y g(x, y) \right)\| \ \le \| \nabla_x g(x, y)\|+ \|\nabla_y g(x, y) \| \ \le \  {\frak m}_{\tilde{f}}(x) + 2 |{\tilde{f}}(x)|.
\end{eqnarray*}
This, together with Claims~\ref{Claim4}~and~\ref{Claim5}, yields
\begin{eqnarray*}
{\frak m}_{\tilde{f}}(x) + 2 |{\tilde{f}}(x)| &\ge& c_1c |g(x, y)|^{1 - \frac{1}{\mathscr R(n + p, 2d + 2)}} \ = \ c_1c |{\tilde{f}}(x)|^{1 - \frac{1}{\mathscr R(n + p, 2d + 2)}}.
\end{eqnarray*}
Note that $\tilde{f}(\bar{x}) = 0.$ Hence, diminishing $\epsilon,$ if necessary, we may assume that
$$|{\tilde{f}}(x)|^{\frac{1}{\mathscr R(n + p, 2d + 2)}} < \frac{c_1c}{4} \quad \textrm{ for all } \quad x \in \mathbb{B}^n(\bar{x}, \epsilon).$$
Consequently, we obtain
\begin{eqnarray*}
{\frak m}_{\tilde{f}}(x) & \ge & \left (c_1c - 2 |{\tilde{f}}(x)|^{\frac{1}{\mathscr R(n + p, 2d + 2)}} \right) |{\tilde{f}}(x)|^{1 - \frac{1}{\mathscr R(n + p, 2d + 2)}} \\
&\ge& \frac{c_1c}{2} |{\tilde{f}}(x)|^{1 - \frac{1}{\mathscr R(n + p, 2d + 2)}},
\end{eqnarray*}
which completes the proof.
\end{proof}

\begin{claim} \label{Claim7}
We have for any $v \in \hat{\partial} {{f}}(x),$
$$2 f(x) v \in \hat{\partial} {\tilde{f}}(x).$$
In particular, if $f(x) > 0$ then
$${\frak m}_{{f}}(x) \ge \frac{1}{2f(x)} {\frak m}_{\tilde{f}}(x).$$
\end{claim}
\begin{proof}
Take arbitrary $v \in \hat{\partial} {{f}}({x}).$ By definition, for every $\epsilon > 0$, there exists $\delta > 0$ such that
$${{f}}({x} + h) - {{f}}({x}) - \langle v, h\rangle \ge - \epsilon \|h\|, \quad \textrm{ for all } \quad h \in \mathbb{B}^n(0, \delta).$$
Let $h \in \mathbb{B}^n(0, \delta).$ It follows from the fact that $f \ge 0$ that
\begin{eqnarray*}
{{f^2}}({x} + h) - {{f^2}}({x}) - \left(f( x + h) + f( x) \right) \langle v, h\rangle & \ge &  - \left (f( x + h) + f( x) \right)  \epsilon \|h\|.
\end{eqnarray*}
Hence
\begin{eqnarray*}
{\tilde{f}}({x} + h) - {\tilde{f}}({x}) - \langle 2f(x) v, h\rangle
& \ge & \left( f(x + h) - f(x) \right) \langle v, h\rangle - \left ( f(x + h) + f(x) \right )  \epsilon \|h\|.
\end{eqnarray*}
Note that $f$ is locally Lipschitz by Claim~\ref{Claim1}, so we have
\begin{equation*}
\begin{array}{lll}
\displaystyle \lim_{\|h\|\to 0}\frac{{\tilde{f}}({x} + h) - {\tilde{f}}({x}) - \langle 2f(x) v, h\rangle }{\|h\|}\\
\displaystyle \ge\lim_{\|h\|\to 0}\frac{\left( f(x + h) - f(x) \right) \langle v, h\rangle - \left ( f(x + h) + f(x) \right )  \epsilon \|h\|}{\|h\|}\\
\displaystyle =\lim_{\|h\|\to 0}\frac{- \left ( f(x + h) + f(x) \right )  \epsilon \|h\|}{\|h\|}=-2f(x)\epsilon.
\end{array}
\end{equation*}
Letting $\epsilon \to 0$ yields $\lim_{\|h\|\to 0}\frac{{\tilde{f}}({x} + h) - {\tilde{f}}({x}) - \langle 2f(x) v, h\rangle }{\|h\|}\ge 0$, and so $2 f(x) v \in \hat{\partial} {\tilde{f}}(x).$
\end{proof}

Now, we are in position to finish the proof of Theorem~\ref{NonSmoothTheorem}.
\begin{proof}[Proof of Theorem~\ref{NonSmoothTheorem}]
By Claim~\ref{Claim6}, there exist some positive constants $c$ and $\epsilon$ such that
\begin{equation*}
\frak m_{\tilde{f}}(x)\ge c\, |{\tilde{f}}(x)|^{1 - \frac{1}{\mathscr R(n + p, 2d + 2)}} \quad \text{ for all } \quad x \in \mathbb{B}^n(\bar{x}, \epsilon).
\end{equation*}
Therefore, by Claim~\ref{Claim7}, we have if $f(x) > 0$ then
\begin{eqnarray*}
{\frak m}_{{f}}(x)
\ \ge \ \frac{1}{2f(x)} {\frak m}_{\tilde{f}}(x) & \ge &
\frac{c}{2f(x)} \left|[f(x)]^2-[f(\bar x)]^2\right|^{1 - \frac{1}{\mathscr R(n + p, 2d + 2)}} \\
&  = & \frac{c\,|f(x)+f(\bar x)|^{1 - \frac{1}{\mathscr R(n + p, 2d + 2)}}}{2f(x)}|f(x)-f(\bar x)|^{1 - \frac{1}{\mathscr R(n + p, 2d + 2)}}.
\end{eqnarray*}
If $f(\bar x)=0$, by continuity and by shrinking $\epsilon$ if necessary, we may assume that $f(x)\le 1$ for $x \in \mathbb{B}^n(\bar{x},\epsilon)$, hence $${\frak m}_{{f}}(x) \ \ge \ \frac{c}{2} \left|f(x)\right|^{1 - \frac{2}{\mathscr R(n + p, 2d + 2)}} \ \ge \ \frac{c}{2} \left|f(x)\right|^{1 - \frac{1}{\mathscr R(n + p, 2d + 2)}}.$$
If $f(\bar x)\not=0$, then by shrinking $\epsilon$ if necessary, we may assume that $f(x)\not=0$ for $x \in \overline{\mathbb{B}}^n(\bar{x},\epsilon)$. Let $\tilde c:=\inf_{x \in \overline{\mathbb{B}}^n(\bar{x},\epsilon)}\frac{c\,|f(x)+f(\bar x)|^{1 - \frac{1}{\mathscr R(n + p, 2d + 2)}}}{2f(x)}>0$. Then ${\frak m}_{{f}}(x) \ge \tilde c|f(x)-f(\bar x)|^{1 - \frac{1}{\mathscr R(n + p, 2d + 2)}}.$ Theorem~\ref{NonSmoothTheorem} is proved.
\end{proof}

\begin{remark}{\rm
From the proof of Theorem~\ref{NonSmoothTheorem}, it is clear that if $f(\bar x)=0$, then there exist some positive constants $c$ and $\epsilon$ such that
\begin{equation}\label{NonSmoothLojasiewiczBetter}
\frak m_f(x)\ge c\, |f(x)|^{1 - \frac{2}{\mathscr R(n + p, 2d + 2)}} \quad \text{ for all } \quad x \in {\mathbb{B}}^n(\bar{x},\epsilon).
\end{equation}}
\end{remark}


\section{Local \L ojasiewicz inequality and local separation of semialgebraic sets} \label{LocalLojasiewicz}

The proof of Theorem~\ref{NonSmoothTheorem} allows us to deduce the following local \L ojasiewicz inequality for the smallest singular value function.

\begin{proposition}\label{CompactErrorBound}
Let $F$ and $f$ be as in Theorem~\ref{NonSmoothTheorem}. Then for any compact set $K \subset\Bbb R^n,$ there exists a constant $c > 0$ such that
\begin{equation}\label{Eq7}
c\, \mathrm{dist}(x, S_F) \le \Big(f(x)\Big)^{\frac{2}{\mathscr R(n+p, 2d+2)}} \quad \textrm{ for all } \quad x\in K,
\end{equation}
where $S_F := \{x \in {\Bbb R}^n \ : \ f(x) =  0\}$ and $\mathscr R(\cdot, \cdot)$ is defined by~(\ref{Eq1}).
\end{proposition}
\begin{proof}

Since $K$ is compact, we can cover $K$ by finite open balls $\Bbb B^n(\bar x_i,\epsilon_i), i=1,\ldots, N,$ such that:
\begin{itemize}
\item Either $\bar x_i\in S_F$ or $\bar{\Bbb B}^n(\bar x_i,\epsilon_i)\cap S_F=\emptyset$;
\item If $\bar x_i\in S_F$ then Inequality~(\ref{NonSmoothLojasiewiczBetter}) holds in $\Bbb B^n(\bar x_i,2\epsilon_i)$.
\end{itemize}
It is clear that by taking $c$ small enough, Inequality~(\ref{Eq7}) holds for all $x\in \Bbb B^n(\bar x_i,\epsilon_i)$ with $\bar{\Bbb B}^n(\bar x_i,\epsilon_i)\cap S_F=\emptyset$ since $\inf_{x\in \Bbb B^n(\bar x_i,\epsilon_i)}[f(x)]_+=\inf_{x\in \bar{\Bbb B}^n(\bar x_i,\epsilon_i)}[f(x)]_+ > 0.$ On the other side, by Lemma~\ref{Lemma22}, Inequality~(\ref{Eq7}) holds for all $x\in \Bbb B^n(\bar x_i,\epsilon_i)$ with $\bar x_i\in S_F$. The proposition follows.
\end{proof}

Another consequence of Theorem~\ref{NonSmoothTheorem} is the following local separation of semialgebraic sets associated to  smallest singular value functions with an explicit exponent, the first and general version go back to \L ojasiewicz~\cite{Lojasiewicz1965} without any precision on the exponent.

\begin{corollary}\label{Separation}
Let $F \colon  \Bbb R^n \rightarrow \mathscr M(p_1,q_1),\ x\mapsto F(x)=(f_{ij}(x)),$ and $G \colon \Bbb R^n \rightarrow \mathscr M(p_2,q_2),\ x\mapsto G(x)=(g_{kl}(x)),$ be two polynomial matrices with $p_1\le q_1$ and $p_2\le q_2$. Let $f$ and $g$ be the corresponding smallest singular value functions. Set
$$S_F := \{x \in {\Bbb R}^n \ : \ f(x) = 0\} \quad \textrm{ and } \quad   S_G := \{x \in {\Bbb R}^n \ : \ g(x) =  0\},$$
and assume that $S_F\cap S_G \ne\emptyset.$ Then for any compact set $K \subset\Bbb R^n,$ there exists a constant $c>0$ such that 
$$c\, \mathrm{dist}(x,S_F\cap S_G) \le \Big (\mathrm{dist}(x,S_F)+\mathrm{dist}(x,S_G) \Big)^{\frac{2}{\mathscr R(n+p_1+p_2, 2d+2)}} \quad \textrm{ for all } \quad x \in K,$$
where $\displaystyle d := \max\{\deg f_{i j},\deg g_{k l}:\ i = 1, \ldots, p_1,\ j = 1, \ldots,\ q_1,\ k = 1, \ldots, p_2,\ l = 1, \ldots, q_2\}.$
\end{corollary}
\begin{proof} Let
$$a(x,y):=\left\|\sum_{i=1}^{p_1}y_iF_i(x)\right\|^2 \ \ \text{ and }\ \ b(x,z):=\left\|\sum_{j=1}^{p_2}z_jG_j(x)\right\|^2$$
where $F_i$ and $G_j$ are, respectively, the $i^\text{th}$ row of $F$ and the $j^\text{th}$ row of $G$.
By Theorem~\ref{NonSmoothTheorem}, we have
$$\tilde f(x):=f^2(x)=\min_{y\in\Bbb R^{p_1},\|y\|=1}a(x,y) \ \ \text{ and }\ \ \tilde g(x):=g^2(x)=\min_{z\in\Bbb R^{p_2},\|z\|=1}b(x,z).$$
So
$$S_F\cap S_G=\{x \in {\Bbb R}^n \ : \ f(x) = g(x) = 0\}=\{x \in {\Bbb R}^n \ : \ f(x) + g(x) = 0\}.$$
Set $E_1(x):=\{y\in\Bbb S^{p_1-1}:\tilde f(x)=a(x,y)\}$ and $E_2(x):=\{z\in\Bbb S^{p_2-1}:\tilde g(x)=b(x,z)\}$. Let $h(x):=\tilde f(x)+\tilde g(x)$, then $S_F\cap S_G=\{x \in {\Bbb R}^n \ : \ h(x) = 0\}$. Let $\bar x\in S_F\cap S_G$, following up the steps of the proof of Theorem~\ref{NonSmoothTheorem}, the reader may check the followings:
\begin{enumerate}
\item[(a)] For all $x\in\Bbb R^n,\ y\in E_1(x),\ z\in E_2(x)$,
\begin{enumerate}
\item[(a1)] $h(x)=a(x,y)+b(x,z),$
\item[(a2)] $\hat \partial h(x) \subset \{\nabla_x(a(x,y)+b(x,z))\}$, so 
$$\frak m_h(x) \ge \inf_{(y', z') \in E_1(x) \times E_2(x)} \|\nabla_x(a(x, y') + b(x, z'))\|,$$

\item[(a3)] $\nabla_ya(x,y)=2\tilde f(x)y$ and $\nabla_zb(x,z)=2\tilde g(x)z$.
\end{enumerate}
 \item[(b)] There exist some positive constants $c$ and $\epsilon'$ such that
 $$\|\nabla_x(a(x,y)+b(x,z))\|\ge c\|a(x,y)+b(x,z)\|^{1-\frac{1}{\mathscr R(n+p_1+p_2,2d+2)}}$$
 for all $x\in\Bbb B^n(\bar x,\epsilon')$ and all $(y,z)\in\Bbb R^{p_1}\times\Bbb R^{p_2}$ with $\mathrm{dist}((y,z),E_1(\bar x)\times E_2(\bar x))<\epsilon'$.
 \item[(c)] For each $\epsilon'>0$, there exists a positive constant $\epsilon<\epsilon'$ such that for all $x\in \Bbb B^n(\bar x,\epsilon)$ and all $(y,z)\in E_1(x)\times E_2(x)$, we have
 $$\mathrm{dist}((y,z),E_1(\bar x)\times E_2(\bar x))<\epsilon'.$$
 \item[(d)] There exist some positive constants $c$ and $\epsilon$ such that
 $$\frak m_h(x)\ge c|h(x)|^{1-\frac{1}{\mathscr R(n+p_1+p_2,2d+2)}}$$
 for all $x\in\Bbb B^n(\bar x,\epsilon).$
\end{enumerate}
By the same proof as Proposition~\ref{CompactErrorBound}, there exists a constant $c>0$ such that for all $x\in K$, we have
$$c\, \mathrm{dist}(x, \{h=0\}) \le \Big(h(x)\Big)^{\frac{1}{\mathscr R(n+p_1+p_2, 2d+2)}}.$$
Therefore
\begin{eqnarray}\label{Eq8}
c\, \mathrm{dist}(x, S_F\cap S_G) \le \Big(f^2(x)+g^2(x)\Big)^{\frac{1}{\mathscr R(n+p_1+p_2, 2d+2)}}\le \Big(f(x)+g(x)\Big)^{\frac{2}{\mathscr R(n+p_1+p_2, 2d+2)}}.
\end{eqnarray}
Since $K$ is compact, $M := \max_{x\in K}\{\textrm{dist}(x,S_F),\textrm{dist}(x,S_G)\}<+\infty$ and $\widetilde{K} := K + M \overline{\mathbb{B}}^n$ is a compact set. Note that the functions $x \mapsto f(x)$ and $x \mapsto g(x)$ are locally Lipschitz, so are globally Lipschitz on the compact set $\widetilde{K}$. Thus there exists a constant $L>0$ such that for all $x,x'\in \widetilde{K}$, we have
$$|f(x)-f(x')|\le L\|x-x'\| \quad \textrm{ and } \quad |g(x)-g(x')|\le L\|x-x'\|.$$
Now for each $x\in K$, there exist $x'\in S_F$ and $x''\in S_G$ such that
$$\textrm{dist}(x,S_F) = \|x-x'\| \quad \textrm{ and } \quad \textrm{dist}(x, S_G) = \|x-x''\|.$$
It is clear that $x',x'' \in \widetilde{K}.$ Hence
\begin{eqnarray*}
|f(x)| &=&  |f(x)-f(x')| \le L\|x-x'\|  = L \, \mathrm{dist}(x,S_F), \\
|g(x)| &=&  |g(x)-g(x'')|\le L\|x-x''\| = L\, \mathrm{dist}(x,S_G).
\end{eqnarray*}
These inequalities, together with Inequality~(\ref{Eq8}), imply the corollary.
\end{proof}


The next result establishes a sharpen version of \L ojasiewicz's factorization lemma for smallest singular value functions.
\begin{corollary}\label{factorization}
Let $F \colon  \Bbb R^n \rightarrow \mathscr M(p_1,q_1),\ x\mapsto F(x)=(f_{ij}(x)),$ $G \colon \Bbb R^n \rightarrow \mathscr M(p_2,q_2),\ x\mapsto G(x)=(g_{kl}(x)),$ and $H \colon \Bbb R^n \rightarrow \mathscr M(p_3,q_3),\ x\mapsto H(x)=(h_{st}(x)),$ be some polynomial matrices with $p_1\le q_1,\ p_2\le q_2$ and $p_3\le q_3$. Let $f(x), g(x),$ and $h(x)$ be the corresponding smallest singular value functions of $F(x), G(x),$ and $H(x)$. Assume that $K := \{x \in {\Bbb R}^n \ : \ h(x) =  0\}$ is a compact set and that
\begin{eqnarray*}
\{x \in K \ : \ f(x) = 0 \} & \subset & \{x \in K \ : \  g(x) = 0\}.
\end{eqnarray*}
Then there is a constant $c > 0$ such that
$$g(x) \le c\,\Big(f(x)\Big)^{\frac{2}{\mathscr R(n + p_1 + p_3 , 2d + 2)}}, \quad \textrm{for all } \quad x\in K,$$
where $\displaystyle d := \max_{i = 1, \ldots, p_1,\ j = 1, \ldots, q_1,\ s = 1, \ldots, p_3,\ t = 1, \ldots, q_3}\{\deg f_{i j},\deg h_{s t}\}.$
\end{corollary}
\begin{proof} The problem is trivial if $\{x \in K \ : \ f(x) = 0 \}=\emptyset$ so assume the contrary. Let
$${\mathcal A} := \{x \in K \ : \ f(x) = 0\} = \{x \in \mathbb{R}^n \ : f(x) = h(x) = 0\}.$$
Similar to Inequality~(\ref{Eq8}) in the proof of Corollary~\ref{Separation}, we have
\begin{eqnarray*}
\textrm{dist}(x, {\mathcal A}) & \le & c_0\Big(f(x)+ h(x)\Big)^{\frac{2}{\mathscr R(n + p_1 + p_3 , 2d + 2)}} \  = \ c_0\Big(f(x)\Big)^{\frac{2}{\mathscr R(n + p_1 + p_3 , 2d + 2)}},
\end{eqnarray*}
for all $x \in K,$ where $c_0$ is a positive constant. Let $M:=\max_{x\in K}\textrm{dist}(x,\{g=0\})<+\infty$ and $\widetilde K := K + M \overline{\mathbb{B}}^n.$ The function $g$ is locally Lipschitz, thus, is globally Lipschitz on $\widetilde K$, i.e., there is a constant $L > 0$ such that $|g(x) - g(y)| \le L \|x - y\|$ for all $x, y \in \widetilde K.$

Now take any $x \in K.$ Clearly, there exists a point $y \in \widetilde K$ such that $g(y) = 0$ and $\textrm{dist}(x, \{g = 0\}) = \|x - y\|.$ Therefore,
\begin{eqnarray*}
g(x) & = & |g(x) - g(y)| \ \le \ L\|x - y\| \ = \  L\, \mathrm{dist}\big(x, \{g = 0\} \big) \\
&\le& L \, \mathrm{dist}\big(x, {\mathcal A} \big) \ \le \ L c_0\Big(f(x)\Big)^{\frac{2}{\mathscr R(n + p_1 + p_3 , 2d + 2)}}.
\end{eqnarray*}
This completes the proof of the corollary.
\end{proof}
\begin{remark}{\rm The statement of Corollary~\ref{factorization} still holds in the case $g \colon K \rightarrow \mathbb{R}$ is a locally Lipschitz function.}
\end{remark}


\section{Global \L ojasiewicz inequality and global separation of semialgebraic sets} \label{GlobalLojasiewicz}

In this section, we provide a global separation of semialgebraic sets and a global \L ojasiewicz inequality with explicit exponents for the case of smallest singular value functions.

\begin{corollary}\label{GlobalSeparation}
Let $F \colon \Bbb R^n \rightarrow \mathscr M(p_1,q_1),\ x\mapsto (f_{ij}(x))$ and $G \colon \Bbb R^n \rightarrow \mathscr M(p_2,q_2),\ x \mapsto (g_{kl}(x))$ be two polynomial matrices with $p_1\le q_1$ and $p_2\le q_2$. Let $f$ and $g$ be the corresponding smallest singular value functions.  Set
$$S_F := \{x \in {\Bbb R}^n \ : \ f(x) = 0\} \quad \textrm{ and } \quad   S_G := \{x \in {\Bbb R}^n \ : \ g(x) = 0\}$$
and assume that $S_F\cap S_G \ne\emptyset.$ Then there exists a constant $c>0$ such that
\begin{eqnarray*}
c \left( \frac{{\rm dist}(x,S_F\cap S_G)}{1  + \|x\|^2} \right)^{\frac{\mathscr R(n + p_1 + p_2 , 2d + 2)}{2}}  & \le &
\mathrm{dist}(x,S_F) + \mathrm{dist}(x,S_G) \quad \textrm{ for all } \quad x \in \mathbb{R}^n,
\end{eqnarray*}
where $\displaystyle d := \max_{i = 1, \ldots, p_1,\ j = 1, \ldots,\ q_1,\ k = 1, \ldots, p_2,\ l = 1, \ldots, q_2}\{\deg f_{i j},\deg g_{k l}\}.$
\end{corollary}
\begin{proof}
The proof follows the same lines of that of~\cite[Theorem 2]{Kurdyka2014}, by using Corollary~\ref{Separation} instead of~\cite[Corollary~8]{Kurdyka2014}. Note that the arguments of the proof of \cite[Theorem 2]{Kurdyka2014} also hold for semialgebraic sets, the assumption of algebraicity is only needed for the application of \cite[Corollary 8]{Kurdyka2014}.
\end{proof}

\begin{remark} {\rm 
Corollary~\ref{GlobalSeparation} can be also obtained by applying \cite[Theorem 1.1]{Kurdyka2016} but the exponent will be different.
}\end{remark}

Next we state a global \L ojasiewicz inequality for smallest singular value functions (compare~\cite[Theorem7]{Solerno1991}):

\begin{corollary}\label{GlobalKollar}
Let $F \colon  \Bbb R^n \rightarrow \mathscr M(p,q),\ x\mapsto F(x)=(f_{ij}(x)),$ be a polynomial matrix with $p\le q$, and $f$ be the corresponding smallest singular value function. Assume that $S_F := \{x \in {\Bbb R}^n \ : \ f(x) =  0\} \ne \emptyset.$
Then for some constant $c > 0,$
\begin{eqnarray*}
c\, \left ( \frac{ \mathrm{dist}(x, S_F)}{1 + \|x\|^2} \right)^{\frac{\mathscr R(n + p + 2 , 4d + 2)}{4}} & \le & f(x) \quad \textrm{ for all } \quad x \in \mathbb{R}^n,
\end{eqnarray*}
where $\displaystyle d := \max_{i = 1, \ldots, p,\ j = 1, \ldots, q} \deg f_{i j}.$
\end{corollary}
\begin{proof}
Define the symmetric polynomial matrices $\widetilde{F} \colon  {\Bbb R^n} \times \mathbb{R} \rightarrow \mathscr M(p,p)$ and
$\widetilde{G} \colon  {\Bbb R^n} \times \mathbb{R} \rightarrow \mathscr M(1,1)$ by
\begin{eqnarray*}
\widetilde{F} (x, y) &:=& F(x)F^T(x) - yI_p \quad \textrm{ and } \quad
\widetilde{G} (x, y) \ := (y)\
\end{eqnarray*}
for $x \in \mathbb{R}^n$ and $y \in \mathbb{R},$ where $I_p$ denotes the unit matrix of order $p.$ Denote by $\tilde f$ and $\tilde g$ the corresponding smallest singular value functions of $\widetilde F$ and $\widetilde G$. Let
$$S_{\widetilde{F}} := \{(x, y) \in {\Bbb R}^n \times {\Bbb R}\ : \ \tilde{f}(x, y) = 0\}\ \ \text{ and }\ \ S_{\widetilde{G}} := \{(x, y) \in {\Bbb R}^n \times {\Bbb R} \ : \ \tilde{g}(x, y) = 0\}.$$
Clearly,
$$S_{\widetilde{F}}=\{(x, y) \in {\Bbb R}^n \times {\Bbb R} \ : \ y \text{ is an eigenvalue of } F(x)F^T(x)\} \ \text{ and }\ \ S_{\widetilde{G}} = \mathbb{R}^n \times \{0\},$$
so $S_{\widetilde{F}} \cap S_{\widetilde{G}}=S_F\times\{0\}.$ By Corollary~\ref{GlobalSeparation}, there exists a constant $c > 0$ such that
\begin{eqnarray*}
c\, \left(\frac{\mathrm{dist}(z, S_{\widetilde{F}} \cap S_{\widetilde{G}})}{1 + \|z\|^2} \right)^{\frac{\mathscr R(n + p + 2 , 4d + 2)}{2}} &\le& \mathrm{dist}(z, S_{\widetilde{F}}) + \mathrm{dist}(z, S_{\widetilde{G}})
\end{eqnarray*}
for all $z := (x, y) \in \mathbb{R}^n \times \mathbb{R}.$ Now it is sufficient to consider $x \in \mathbb{R}^n$ satisfying $f(x) > 0.$ It is clear that $\mathrm{dist}((x,0), S_{\widetilde{G}})=0$. Moreover $\mathrm{dist}((x,0), S_{\widetilde{F}} \cap S_{\widetilde{G}})= \mathrm{dist}(x,S_F)$. Note that $(x, f^2(x)) \in S_{\widetilde{F}}.$ Thus
\begin{eqnarray*}
\mathrm{dist}((x,0), S_{\widetilde{F}}) & \le & \|(x, 0) - (x, f^2(x))\| \ = \ f^2(x).
\end{eqnarray*}
The corollary follows.
\end{proof}

As a direct consequence of Corollary~\ref{GlobalKollar}, we obtain the following result
(see~\cite{Kollar1988, Kurdyka2014}):

\begin{corollary}
Let $F \colon  \Bbb R^n \rightarrow \mathscr M(p,q),\ x\mapsto F(x)=(f_{ij}(x)),$ be a polynomial matrix with $p\le q$, and $f$ be the corresponding smallest singular value function. Assume that $S_F := \{x \in {\Bbb R}^n \ : \ f(x) =  0\}$ is a nonempty compact set.
Then there are some constants $c > 0$ and $R > 0$ such that
\begin{eqnarray*}
c\|x\|^{-\frac{\mathscr R(n + p + 2 , 4d + 2)}{4}} & \le & f(x),  \quad \textrm{ for all } \quad \|x\| \ge R,
\end{eqnarray*}
where $\displaystyle d := \max_{i = 1, \ldots, p,\ j = 1, \ldots, q} \deg f_{i j}.$
\end{corollary}
\begin{proof}
Indeed, since the set $S_F$ is compact, we can find some positive constants $c_1$ and $c_2$ satisfying the following inequality
$$c_1 \|x\| \le \textrm{dist}(x, S_F) \le c_2 \|x\| \quad \textrm{ for } \quad \|x\| \gg 1.$$
Combining this with Corollary~\ref{GlobalKollar} yields the desired conclusion.
\end{proof}

\section{Global \L ojasiewicz inequality and goodness at infinity} \label{TameLojasiewicz}

In this part, we give a global version of \L ojasiewicz inequality with explicit exponent in the case of tame singular value functions.

\begin{definition}{\rm
We say that the polynomial matrix $F \colon  \Bbb R^n \rightarrow \mathscr M(p,q)$ is {\em good at infinity} if its corresponding singular value function $f$ is tame i.e., there exist some constants $c > 0$ and $R > 0$ such that
$${\frak m}_f(x) \ge c \quad \textrm{ for all } \quad \|x\| \ge R.$$
}\end{definition}

\begin{proposition}\label{HolderTypeTheorem}
Let $F \colon  \Bbb R^n \rightarrow \mathscr M(p,q),\ x\mapsto F(x)=(f_{ij}(x)),$ be a polynomial matrix with $p\le q$, and $f$ be the corresponding smallest singular value function. Assume that $S_F := \{x \in {\Bbb R}^n \ : \ f(x) =  0\} \ne \emptyset$ and that $F$ is good at infinity. Then the set $S_F$ is compact and there exists a constant $c > 0$ such that
$$c\, {\rm dist}(x, S_F) \ \le \ \Big(f(x)\Big)^{\frac{2}{\mathscr R(n+p,2d+2)}} + f(x) \quad \textrm{ for all } \quad x \in \mathbb{R}^n,$$
where $d:=\max_{i = 1, \ldots, p,\ j = 1, \ldots, q}\deg f_{i j}.$
\end{proposition}

First of all, we prove the following claim.
\begin{claim}\label{Claim8}
Assume that there exist some constants $c > 0$ and $R > 0$ such that
\begin{equation*}
{\frak m}_{f} (x) \ge c \quad \textrm{ for all } \quad \|x\| \ge R.
\end{equation*}
Let $s \in S_F.$ Then
\begin{equation*}
\frac{c}{2}\, \mathrm{dist}(x, S_F)\le f(x) \quad \textrm{ for all } \quad \|x\|\ge 3R + 2\|s\|.
\end{equation*}
\end{claim}
\begin{proof}
We argue by contradiction. Suppose that the conclusion is false. Then there exists $\bar x \in \mathbb{R}^n$ such that $$\|\bar x\|\ge 3R + 2\|s\| \quad \textrm{ and } \quad f(\bar x)<\frac{c}{2}\textrm{dist}(\bar x, S_F).$$
Clearly $\bar x\not\in S_F$. Set $K:=\{x \in \mathbb{R}^n \ : \ \|x\|\ge R\}.$ Note that $\inf_K f(x)\ge 0$, so
$$f(\bar{x}) < \inf_K f(x) + \frac{c}{2}\textrm{dist}(\bar{x}, S_F).$$
By applying Ekeland variational principle~\cite{Ekeland1979} to the function $f(x)$ on the closed set $K$ with the data $\epsilon := \frac{c}{2}\textrm{dist}(\bar x, S_F) > 0$ and $\lambda := \frac{2\textrm{dist}(\bar x, S_F)}{3} > 0$, we assert that there is $\bar y\in K$ such that $\|\bar y - \bar x\| < {\lambda}$ and that $\bar y$ minimizes the function
$$K \rightarrow \mathbb{R}, \quad x \mapsto f(x) + \frac{\epsilon}{\lambda}\|x-\bar y\|.$$
It follows that
\begin{eqnarray*}
\|\bar y\| & \ge & \|\bar x\| - \|\bar y - \bar x\| >  \|\bar x\| - \frac{2}{3}\textrm{dist}(\bar x, S_F) \\
& \ge & \|\bar{x}\| - \frac{2}{3}\|\bar{x} - s\| \ge \|\bar{x}\| - \frac{2}{3}(\|\bar x\| + \|s\|)
= \frac{1}{3}(\|\bar{x} - 2 \|s\|) \ge R.
\end{eqnarray*}
Thus $\bar y$ is an interior point of $K.$ Then we deduce from Claim~\ref{Lemma21} that
$$0\in \partial f(\bar y) + \frac{\epsilon}{\lambda}\Bbb B^n.$$
By the definition of the function $\frak m_{f},$ it follows easily that
$$\frak m_{f} (\bar y)\leq \frac{\epsilon}{\lambda}.$$
Since $\bar x \not \in S_F$ and $\|\bar y - \bar x\| < \lambda = \frac{2}{3}\textrm{dist}(\bar x, S_F)$, we have
$\bar{y} \not \in S_F$ and so $f(\bar y)>0.$ Therefore
$$\frak m_{f} (\bar y)\leq \frac{\epsilon}{\lambda} = \frac{3c}{4} < c,$$
which is a contradiction.
\end{proof}

Now, we are in position to finish the proof of Proposition~\ref{HolderTypeTheorem}.

\begin{proof} [Proof of Proposition~\ref{HolderTypeTheorem}]
By the assumption, there exist some constants $c_1 > 0$ and $R > 0$ such that
$${\frak m}_f(x) \ge c_1 \quad \textrm{ for all } \quad \|x\| \ge R.$$
Since the function $f$ is non-negative on $\mathbb{R}^n,$ it follows from Claim~\ref{Lemma21}(i) that if $x \in S_F$ then $0 \in \partial f(x)$ and so ${\frak m}_f(x) = 0.$ Consequently, $S_F \subset \{x \in \mathbb{R}^n \ : \ \|x\| \le R \}$ and hence $S_F$ is a compact set.

Now, let us fix a point $s$ in $S_F.$ Due to Claim~\ref{Claim8}, we obtain
\begin{equation}\label{nu1}
\frac{c_1}{2}\, \mathrm{dist}(x, S_F) \le f(x), \quad \text{ for all } \quad  \|x\|\ge 3R + 2\|s\|.
\end{equation}
On the other hand, thanks to Proposition~\ref{CompactErrorBound}, we get a constant $c_2 > 0$ satisfying
\begin{equation}\label{nu2}
c_2 \, \mathrm{dist}(x, S_F)\leq \Big(f(x)\Big)^{\frac{2}{\mathscr R(n+p,2d+2)}}, \quad \textrm{ for all }  \ \|x\| \le 3R + 2\|s\|.
\end{equation}
Let $c := \min \{\frac{c_1}{2}, c_2\} > 0.$  Taking account of Inequalities~(\ref{nu1})~and~(\ref{nu2}), we obtain
$$c \, \mathrm{dist}(x,S_F)\leq \Big(f(x)\Big)^{\frac{2}{\mathscr R(n+p,2d+2)}} + f(x), \quad \textrm{for all } \quad x \in \Bbb R^n,$$
as it was to be shown.
\end{proof}

\noindent\textbf{Acknowledgment.} This research was partially performed while the authors had been visiting at Vietnam Institute for Advanced Study in Mathematics (VIASM). The authors would like to thank the Institute for hospitality and support.


\begin{thebibliography}{99}
\bibitem{Absil2006}
P.-A. Absil, K. Kurdyka, {\em On the stable equilibrium points of gradient systems}, Systems Control Lett., {\bf 55} (7) (2006), 573--577.

\bibitem{Acunto2005}
D. D'Acunto, and K. Kurdyka, Explicit bounds for the \L ojasiewicz exponent in the gradient inequality for polynomials, {\em Ann. Polon. Math.} {\bf 87} (2005), 51--61.

\bibitem{Bochnak1998}
J. Bochnak, M. Coste, and M.-F. Roy, {\em Real algebraic geometry}, Vol. 36, Springer, 1998.


\bibitem{Dinh2016}
S. T. \DD inh, and T. S. Ph\d{a}m, {\em \L ojasiewicz-type inequalities with explicit exponents for the largest eigenvalue function of real symmetric polynomial matrices}, {\bf 27} (2) (2016), 1650012 (27 pages).

\bibitem{Dries1996}
L. van den Dries, and C. Miller, {\em Geometric categories and o-minimal structures,} Duke Math. J., {\bf 84} (1996), 497--540.

\bibitem{Ekeland1979}
I. Ekeland, {\em Nonconvex minimization problems,} Bull. A.M.S., No. 1 (1979), 443--474.

\bibitem{Haraux2012}
A. Haraux, {\em Some applications of the \L ojasiewicz gradient inequality}, Commun. Pure Appl. Anal., {\bf 11} (6) (2012), 2417--2427.

\bibitem{Kollar1988}
J. Koll\'ar, {\em Sharp effective Nullstellensatz,} J. Amer. Math. Soc., {\bf 1} (4) (1988), 963--975,

\bibitem{Kurdyka2000}
K. Kurdyka, T. Mostowski, A. Parusinski, {\em Proof of the gradient conjecture of R. Thom}, Ann. of Math. (2), {\bf 152} (3) (2000), 763--792.

\bibitem{Kurdyka2014}
K. Kurdyka and  S. Spodzieja, {\em Separation of real algebraic sets and the \L ojasiewicz exponent,}
Proc. Amer. Math. Soc., {\bf 142} (9) (2014), 3089--3102.

\bibitem{Kurdyka2016}
K. Kurdyka, S. Spodzieja, A. Szlachcinska, {\em Metric properties of semialgebraic mappings}, Preprint. Available from: arXiv:1412.5088.

\bibitem{Lojasiewicz1958}
S. \L ojasiewicz, {\em Division d'une distribution par une fonction analytique de variables r\'eelles,}
C. R. Acad. Sci. Paris, {\bf 246} (1958), 683--686.

\bibitem{Lojasiewicz1965}
S. \L ojasiewicz, {\em Ensembles semi-analytiques,} I.H.E.S, Bures-sur-Yvette, 1965.

\bibitem{Lojasiewicz1984}
S. \L ojasiewicz, {\em Sur les trajectoires du gradient d'une fonction analytique,} Geometry seminars, 1982--1983 (Bologna, 1982/1983),  Univ. Stud. Bologna, (1984), 115--117.

\bibitem{Mordukhovich2006}
B. S. Mordukhovich, {\em Variational analysis and generalized differentiation, I: Basic Theory, II: Applications}, Springer, Berlin, 2006.

\bibitem{Ngai2009}
H. V. Ngai, and M. Thera, {\em Error bounds for systems of lower semicontinuous functions in Asplund spaces}, Math. Program., Ser. B, {\bf 116} (1-2) (2009), 397--427.

\bibitem{Nie2007}
J. Nie, and M. Schweighofer, {\em On the complexity of Putinar's Positivstellensatz}, J. Complexity, {\bf 23} (1) (2007), 135--150.

\bibitem{Rockafellar1998}
R. T. Rockafellar, and R. Wets, {\em Variational analysis,} Grundlehren Math. Wiss., {\bf 317}, Springer, New York, 1998.

\bibitem{Schweighofer2004}
M. Schweighofer, {\em On the complexity of Schm\"udgen's positivstellensatz}, J. Complexity, {\bf 20} (4) (2004), 529--543.

\bibitem{Seidenberg1954}
A. Seidenberg, {\rm A new decision method for elementary algebra,} Ann. of Math. (2), {\bf 60} (1954), 365--374.

\bibitem{Solerno1991}
P.  Solern\'o, {\rm Effective  \L ojasiewicz  inequalities  in  semialgebraic  geometry,} Appl. Algebra Engrg. Comm.  Comput., {\bf 2} (1991), 2--14.

\bibitem{Tarski1931}
A. Tarski, {\rm Sur les ensembles d\'efinissables de nombres r\'eels,} Fund. Math., {\bf 17} (1931), 210--239.

\bibitem{Tarski1951}
A. Tarski, {\rm A decision method for elementary algebra and geometry,} 2nd ed. University of California Press, Berkeley and Los Angeles, Calif., 1951. iii+63 pp.
\end{thebibliography}
\end{document}